\documentclass[11pt]{article}
\usepackage{a4wide}
\usepackage{graphicx}
\usepackage{color}
\usepackage{amsmath}
\usepackage[mathcal]{eucal}
\usepackage{amsthm}
\usepackage{amssymb}
\usepackage{caption}
\usepackage{tikz,pgfmath}
\usepackage{hyperref}
\usepackage{bm}
\usepackage{algorithmic}

\newtheorem{thm}{Theorem}[section]
\newtheorem{defn}[thm]{Definition}
\newtheorem{problem}[thm]{Problem}
\newtheorem{prop}[thm]{Proposition}
\newtheorem{cor}[thm]{Corollary}
\newtheorem{example}[thm]{Example}
\newtheorem{claim}{Claim}

\DeclareMathOperator{\adm}{adm}
\DeclareMathOperator{\ch}{ch}
\DeclareMathOperator{\co}{col}
\DeclareMathOperator{\di}{dist}
\DeclareMathOperator{\gc}{gcol}
\DeclareMathOperator{\sco}{scol}
\DeclareMathOperator{\td}{td}
\DeclareMathOperator{\tw}{tw}
\DeclareMathOperator{\wc}{wcol}
\newcommand{\ssetminus}{\!\smallsetminus\!}

\newcommand{\qitem}[1]{\noindent\leavevmode\hangindent1.5\parindent%
 \noindent\hbox to1.5\parindent{#1\hss}\ignorespaces}

\newcommand{\NN}{\mathbb{N}}

\title{Uniform Orderings for Generalized Coloring Numbers\,\thanks{\,The
    research for this paper was started during a visit of HAK to the London
    School of Economics, and continued during a return visit of JvdH to
    Arizona State University. The authors would like to thank both
    universities for their hospitality and support.}}

\author{Jan van den Heuvel\,\thanks{\,Department of Mathematics, London
    School of Economics and Political Science, London WC2A 2AE, UK.} \ \
  and \ \ H.A. Kierstead\,\thanks{\,School of Mathematical and Statistical
    Sciences, Arizona State University, Tempe, AZ 85287, USA.}}

\begin{document}
\maketitle

{\renewcommand{\thefootnote}{\roman{footnote}}
  \footnotemark[0]
  \footnotetext[0]{\,Email: \texttt{j.van-den-heuvel@lse.ac.uk},
    \texttt{kierstead@asu.edu}.}}%

\vspace{-5mm}
\begin{abstract}
  \noindent
  The generalized coloring numbers $\co_r(G)$ (also denoted by $\sco_r(G)$)
  and $\wc_r(G)$ of a graph~$G$ were introduced by Kierstead and Yang as a
  generalization of the usual coloring number, and have found important
  theoretical and algorithmic applications. For each distance~$r$, these
  numbers are determined by an ``optimal'' ordering of the vertices
  of~$G$. We study the question of whether it is possible to find a single
  ``uniform'' ordering that is ``good'' for all distances~$r$.

  We show that the answer to this question is essentially ``yes''. Our
  results give new characterizations of graph classes with bounded
  expansion and nowhere dense graph classes.

\end{abstract}

\section{Introduction and Main Results}\label{sec1}

\subsection{Coloring Numbers}

All graphs $G=(V,E)$ in this paper are finite, simple and undirected. We
use~$|G|$ for $|V|$. By an \emph{ordering $\sigma$} of a graph we mean a
total ordering of its vertex set, i.e.\ for every $x,y\in V$, $x\ne y$, we
have exactly one of $x<_\sigma y$ or $y<_\sigma x$. The set of all
orderings of~$G$ is denoted $\Pi(G)$ (or just $\Pi$, if the graph is clear
from the context).

For a graph~$G$, $\sigma\in\Pi$ and $x\in V$, let $\co(G,\sigma,x)$ be one
more than the number of neighbors $y\in N_G(x)$ with $y<_\sigma x$. The
\emph{coloring number of $G$}, denoted $\co(G)$, is defined by
\[\co(G)=\min_{\sigma\in\Pi}\,\max_{x\in V}\,\co(G,\sigma,x).\]
In recent terminology, the coloring number of a graph is one more than its
degeneracy; under an older definition of degeneracy they were the
same. Greedily coloring the vertices of~$G$ in an ordering that witnesses
its coloring number, shows that
\[\chi(G)\le \ch(G)\le \co(G),\]
where $\chi(G)$ and $\ch(G)$ denote the chromatic and list chromatic number
of $G$, respectively.

An alternative way to define $\co(G,\sigma,x)$ is as the number of vertices
$y\le_\sigma x$ that have distance at most $1$ from $x$. (Since $x$ has
distance $0$ from itself, we count $x$ in this definition as well, avoiding
having to add ``one more than'' as in our first definition.)  In this paper
we are interested in \emph{generalized coloring numbers}, where we consider
vertices $y\le_\sigma x$ that are at some further distance $r$ from
$x$. These numbers were first introduced in \cite{KY}, after similar
notions were explored by various authors
\cite{MR1198403,K-Comp,45,KT-OGC,Zhu-RAS} in the cases $r=2,4$.

Since there are several choices we can impose on the position of the
internal vertices of a path from~$x$ to~$y$ with respect to an ordering
$\sigma$, we define two variants. Let $r\in\NN\cup\{\infty\}$. For a graph
$G$, ordering $\sigma\in\Pi$ and $x\in V$, we say that a vertex $y$ is
\emph{weakly $r$-reachable} from~$x$ with respect to~$\sigma$ if
$y\le_\sigma x$ and there is an $x,y$-path $P$ with length $|E(P)|\le r$
such that all vertices $p\in V(P)$ satisfy $p\ge_\sigma y$; $y$ is
\emph{strongly $r$-reachable} from $x$ with respect to $\sigma$ if we have
the stronger condition that all $p\in V(P)\ssetminus\{y\}$ satisfy
$p\ge_\sigma x$. Let $W_r[G,\sigma,x]$ be the set of vertices that are
weakly $r$-reachable from $x$ with respect to $\sigma$ and
$S_r[G,\sigma,x]$ be the set of vertices that are strongly $r$-reachable
from $x$ with respect to $\sigma$. Note that $x$ itself is included in both
$W_r[G,\sigma,x]$ and $S_r[G,\sigma,x]$.

The \emph{weak $r$-coloring number of $G$}, denoted $\wc_r(G)$, and the
\emph{strong $r$-coloring number of~$G$}, denoted $\sco_r(G)$, are defined
by\footnote{\,In \cite{KY} strong coloring numbers were just called
  \emph{coloring numbers}, and weak coloring numbers were introduced for
  the purpose of studying (strong) coloring numbers. As weak coloring
  numbers have their own merit, it now seems better to distinguish between
  them by using the terms \emph{strong} and \emph{weak}.}:
\begin{alignat*}{2}
  &\wc_r(G,\sigma)=\max_{x\in V}\bigl|W_r[G,\sigma,x]\bigr|;\quad
  &&\wc_r(G)=\min_{\sigma\in\Pi}\,\wc_r(G,\sigma);\\
  &\sco_r(G,\sigma)=\max_{x\in V}\bigl|S_r[G,\sigma,x]\bigr|;\quad
  && \sco_r(G)=\min_{\sigma\in\Pi}\,\sco_r(G,\sigma).
\end{alignat*}

We obviously have $\co(G)=\wc_1(G)=\sco_1(G)$.

The following easy observations hint at the usefulness of different
versions of coloring numbers. If the vertices of $G$ are colored greedily
so that no vertex $v$ receives the same color as any other vertex in
$S_2[G,\sigma,v]$, then the resulting coloring is an acyclic coloring, so
\[\ch_{\text{a}}(G)\le\sco_2(G),\]
where $\ch_{\text{a}}(G)$ denotes the list acyclic chromatic number of $G$.
If the vertices of $G$ are colored greedily so that no vertex $v$ receives
the same color as any vertex in $W_2[G,\sigma,v]$, then the resulting
coloring is a star coloring, so
\[\ch_{\text{s}}(G)\le\wc_2(G),\]
where $\ch_{\text{s}}(G)$ denotes the star chromatic number of $G$.

As noticed already in \cite{KY}, the two types of generalized coloring
numbers are related by the inequalities
\begin{equation}
\sco_r(G)\le\wc_r(G)\le(\sco_r(G))^r.\label{eq1}
\end{equation}
Thus if one of the generalized coloring numbers is bounded for a class
of graphs (for some~$r$), then so is the other one.

An interesting aspect of generalized coloring numbers is that they can also
be seen as gradations between the coloring number $\co(G)$ and two
important graph invariants, namely the \emph{tree-width $\tw(G)$} and the
\emph{tree-depth $\td(G)$}. (The latter is the minimum height of a
depth-first search tree for a supergraph of $G$ \cite{nevsetvril2006tree}.)
More explicitly, we have the following proposition.

\begin{prop}\label{pro:coltwtd}\mbox{}\\*
  Every graph $G$ satisfies:

  \smallskip
  \qitem{(a)} $\co(G)= \sco_1(G)\le \sco_2(G)\le \ldots\le \sco_\infty(G)=
  \tw(G)+1$;

  \smallskip
  \qitem{(b)} $\co(G)= \wc_1(G)\le \wc_2(G)\le \ldots\le \wc_\infty(G)=
  \td(G)$.
\end{prop}

\noindent
The equality $\sco_\infty(G)=\tw(G)+1$ was first proved in \cite[Section
6]{GKRSS}. The equality $\wc_\infty(G)=\td(G)$ is proved in
\cite[Lemma~6.5]{NO-Sparsity}.

Generalized coloring numbers have been instrumental in the study of sparse
graph classes. Ne\v{s}et\v{r}il and Ossona de Mendez introduced the notion
of \emph{graph classes with bounded expansion}~\cite{NO-I} and the more
general notion of \emph{nowhere dense graph classes}~\cite{NO-ND}. These
concepts generalize those of graph classes with bounded tree-width,
minor-closed classes, bounded degree classes, etc. See the book of
Ne\v{s}et\v{r}il and Ossona~de Mendez~\cite{NO-Sparsity} for a wealth of
information about the properties of these graph classes.

One of the key properties of this classification is that it is remarkably
robust. Not only can results for particular classes that have bounded
expansion (or are nowhere dense) often be generalized to all classes with
that property, but these generalizations often yield new characterizations.
For example, classes with bounded generalized coloring numbers were studied
in \cite{KY} because they had bounded generalized game coloring numbers
(see Section~\ref{sec3} for definitions). Later, Zhu~\cite{Zhu-BGCN} proved
bounds on the generalized coloring numbers that gives the following
characterizations of bounded expansion and nowhere dense classes in terms
of those numbers. We will use these characterizations as definitions.

\begin{defn}\label{def1}\mbox{}
  
  \qitem{(a)} A graph class~$\mathcal{G}$ has \emph{bounded expansion} if
  and only if there exists a function $c:\NN\to\NN$ such that
  $\sco_r(G)\le c(r)$ for all~$r$ and all $G\in\mathcal{G}$.

  \qitem{(b)} A graph class~$\mathcal{G}$ is \emph{nowhere dense} if and
  only if there exists a function $n_0:\mathbb{R}\times\NN\to\NN$ such that
  for every $\epsilon>0$, $r\in\NN$ and $G\in\mathcal{G}$ we have that
  $\sco_r(H)\le|H|^\epsilon$ for all subgraphs $H$ of $G$ with
  $|H|\ge n_0(\epsilon,r)$.
 \end{defn}

\noindent
Note that by the inequalities in~\eqref{eq1} we equally well could have
defined bounded expansion and nowhere dense in terms of the weak coloring
numbers.

Here is a different example demonstrating the surprising power of this
classification of sparse graph classes. Streib and Trotter \cite{ST} proved
that every poset whose cover graph is planar, has dimension bounded by a
function of its height. Then Joret et al.~\cite{JMOWb} used generalized
coloring numbers to prove that every monotone graph class $\mathcal{G}$ is
nowhere dense if and only if for every integer $h\ge1$ and real number
$\epsilon>0$, every $n$-element poset of height at most $h$ whose cover
graph is in $\mathcal{G}$ has dimension $O(n^\epsilon)$.

Generalized coloring numbers are an important tool in the context of
algorithmic sparse graphs theory; see again~\cite{NO-Sparsity}. More
recently they have played a key role in algorithmic results on
model-checking for first-order logic on bounded expansion and nowhere dense
graph classes \cite{DvKralTh,GhKS,Kazana-S}.

\subsection{The Guiding Question}

An obvious question concerning generalized coloring numbers is whether an
ordering that is ``good'' for one distance~$r$ is also ``good'' for a
different distance~$r'$. In fact, this need not be the case: in
Example~\ref{exa:counter} we will show that for all $r,r'\in\NN$ with
$r\ne r'$, there exists a graph~$G$ such that for all $\sigma\in\Pi(G)$
either $\sco_r(G)<\sco_r(G,\sigma)$ or $\sco_{r'}(G)<\sco_{r'}(G,\sigma)$.

The existence of examples as above also has consequences for the many
algorithms that for a graph class~$\mathcal{G}$ with bounded expansion and
some~$r$, use explicitly an ordering~$\sigma$ which shows that
$\sco_r(G)\le c(r)$. It looks as if for every~$r$ a different ordering is
needed.

Given a function $c:\NN\to\NN$, let $\mathcal{G}_c$ be the graph class
defined by: $G\in\mathcal{G}_c$ if and only if $\sco_r(G)\le c(r)$ for all
$r\in\NN$. Then the class $\mathcal{G}_c$ has bounded expansion, and every
class with bounded expansion is contained in $\mathcal{G}_{c'}$ for some
$c'$.

In this paper we investigate the following problem that was raised by
Dvo\v{r}\'ak \cite{Problem}. Kreutzer et al.\ \cite[Section 6]{KrPiRASi}
state that it is ``tempting to conjecture'' that the answer to this problem
is yes.

\begin{problem}\label{prob:Dv}\mbox{}\\*
  Is it true that for all functions $c:\NN\to\NN$, there exists a function
  $c^*:\NN\to\NN$, such that for every graph $G\in\mathcal{G}_c$, there
  exists an ordering $\sigma^*\in\Pi(G)$ such that
  $\sco_r(G,\sigma^*)\le c^*(r)$ for all $r\in\NN$?
\end{problem}

\noindent
The main reason this issue was raised by several people was that for all
known bounds on the generalized coloring numbers on graph classes such as
(topological) minor closed classes, a single ordering of all graphs in the
class gave those bounds for all distances~$r$; see
e.g.~\cite{vdHOdMEA,KrPiRASi}.

\subsection{Results}\label{ssec1.3}

Our main result provides a positive answer for Problem~\ref{prob:Dv}.

\begin{thm}\label{thm:main}\mbox{}\\*
  For any graph~$G$, there exists an ordering $\sigma^*$ of~$G$ such that
  for all $r\in\NN$ we have
  \[\sco_r(G,\sigma^*)\le(2^r+1)\cdot\bigl(\sco_{2r}(G)\bigr)^{4r}.\]
\end{thm}

\medskip\noindent
In the terminology of Problem~\ref{prob:Dv}, this means we can set
$c^*(r)=(2^r+1)\cdot\bigl(c(2r)\bigr)^{4r}$ for all $r$.

We immediately obtain the following new characterizations of graph classes
with bounded expansion and nowhere dense graph classes.

\begin{cor}\label{cor:1}\mbox{}\\*
  A graph class $\mathcal{G}$ has bounded expansion if and only if there
  exists a function $c^*:\NN\to\NN$, such that for every graph
  $G\in\mathcal{G}$ there exists an ordering $\sigma^*(G)$ of $G$ such that
  $\sco_r(G,\sigma^*(G))\le c^*(r)$ for all~$r$.
\end{cor}

\begin{cor}\label{cor:2}\mbox{}\\*
  A graph class $\mathcal{G}$ is nowhere dense if and only if there exists
  a function $n^*_0:\mathbb{R}\times\NN\to\NN$ such that for every
  subgraph~$H$ of a graph $G\in\mathcal{G}$, there exists an ordering
  $\sigma^*(H)$ of~$H$ such that for all $\epsilon>0$ and $r\in\NN$, if
  $|H|\ge n^*_0(\epsilon,r)$, then $\sco_r(H,\sigma^*(H))\le|H|^\epsilon$.
\end{cor}

\noindent
By the definition of the strong coloring number it follows that if $G$ is a
graph with some ordering~$\sigma^*(G)$, then for every subgraph $H$ of $G$,
if we take $\sigma^*(H)$ the ordering of $H$ induced by~$\sigma^*(G)$, we
have $\sco_r(H,\sigma^*(H))\le\sco_r(G,\sigma^*(G))$ for all~$r$. This
means that in Corollary~\ref{cor:1} once we have an ordering $\sigma^*(G)$
for some graph $G\in\mathcal{G}$, for every $H\in\mathcal{G}$ that is a
subgraph of~$G$ we can take the ordering $\sigma^*(H)$ of $H$ induced by
$\sigma^*(G)$. In view of this it is natural to ask whether a similar
statement is possible for the condition in Corollary~\ref{cor:2} for
nowhere dense classes of graphs. In Subsection~\ref{ssec2.2} we will show
that this is in fact not possible.

Theorem~\ref{thm:main} above follows from a technical, more general, result
that deals with different graphs on the same vertex set; see
Section~\ref{sec4}. Another consequence of this more general result is the
following theorem, which may be of independent interest.

\begin{thm}\label{thm:graphU}\mbox{}\\*
  Let $G_1,\ldots,G_k$ be a collection of graphs, all on the same vertex
  set~$V$, and let $r_1,\ldots,r_k\in\NN$. Then there exists a ordering
  $\sigma^*$ of the common vertex set~$V$ such that for all $i=1,\ldots,k$,
  \[\sco_{r_i}(G_i,\sigma^*)\le
    (k+1)\bigl(\wc_{2r_i}(G_i)\bigr)^2\le
    (k+1)\bigl(\sco_{2r_i}(G_i)\bigr)^{4r_i}.\]
\end{thm}

\medskip\noindent
The proof of the general result, which also can be found in
Section~\ref{sec4}, has at its basis arguments developed in \cite{KY,KY2}.

The remainder of this paper is organized as follows. In the next subsection
we give essential terminology and notation. The two classes of examples
referred to earlier can be found in Section~\ref{sec:Example}. In
Section~\ref{sec3} we describe the essential concepts and the result
from~\cite{KY} that provided the inspiration for our proof of the main
theorem. In Section~\ref{sec4} we state and prove our main technical
result, and give the proofs of its corollaries. In the next section we
discuss some algorithmic aspects of our results. We discuss some open
questions in the final section.

\subsection{Terminology and Notation}

Most of our graph theory terminology and notation is standard and can be
found in text books such as \cite{Diestel5}.

If $P=v_1v_2\dots v_n$ is a path, then we call $v_1$ and $v_n$ the
\emph{ends of $P$}. The subpath of $P$ that has ends $a$ and $b$ is denoted
by $aPb$. Finally, $\mathring{P}$ is $P$ minus its ends. The \emph{length}
of a path is the number of edges in it. (So one fewer than the number of
vertices.)

For two vertices $x$ and $y$ in the same component of a graph $G=(V,E)$,
the \emph{distance} $\di_G(x,y)$ between~$x$ and $y$ is the length of a
shortest $x,y$-path in $G$. For $v\in V$, $N_G(v)$ denotes the set of
vertices in~$G$ adjacent to~$v$; $N_G[v]=N_G(v)\cup\{v\}$. For a subset
$X\subseteq V$, $G[X]$ denotes the subgraph of~$G$ induced on the vertex
set $X$.

For a positive integer $k$, we write $[k]=\{1,2,\ldots,k\}$.

If $\sigma$ is an ordering of some set $X$ and $S,T$ are non-empty subsets
of $X$, then by $S<_\sigma T$ we mean that $s<_\sigma t$ for all $s\in S$,
$t\in T$. We abbreviate $\{s\}<_\sigma T$ to $s<_\sigma T$. The element
in~$S$ that is minimum with respect to $\sigma$ is denoted by
$\sigma$-$\min(S)$. The ordering $\sigma_S$ on $S$ \emph{induced
  by~$\sigma$} is the ordering given by: $s_1<_{\sigma_S}s_2$ if and only
if $s_1<_\sigma s_2$, for all $s_1,s_2\in S$.

\section{Examples}\label{sec:Example}

\subsection{Graphs with No ``Good'' Ordering}

The following examples show that in answering Problem~\ref{prob:Dv} we
cannot take $c^*=c$.

\begin{example}\label{exa:counter}\mbox{}\\*
  Let $\varphi$ be the largest solution to $x^2=x+1$ (the golden ratio
  $\varphi=\frac12(1+\sqrt5)\approx 1.62$). For all $r,r'\in\NN$ with
  $r<r'$, there exists a graph $G$ such that for all $\sigma\in\Pi(G)$,
  either
  \[\sco_r(G,\sigma)> .08\bigl(\sco_r(G)\bigr)^\varphi\quad
    \text{or}\quad \sco_{r'}(G,\sigma)>
    .08\bigl(\sco_{r'}(G)\bigr)^\varphi.\]
\end{example}

\begin{proof}
  Fix $t,n\in\NN$ with $4\le t\le n$. Let
  $Z=\{z_i^h\mid i\in[n],\:h\in[t]\}$ be a set of vertices, and partition
  $Z$ into $n$ sets $Z_i=\{z_i^h\mid h\in[t]\}$ of size $t$. We construct
  $G$ by connecting each ordered pair $(Z_i,Z_j)$, $i\ne j$, with
  isomorphic graphs $H_{i,j}$ so that
  $G=\bigcup\{H_{i,j}\mid i,j\in[n],\: i\ne j\}$ and the $H_{i,j}$ are
  pairwise disjoint except for their ends in $Z$. In particular, the
  sets~$Z_i$ and~$Z_j$ are connected by both $H_{i,j}$ and~$H_{j,i}$.

  For all $h\in[t]$ and $i,j\in[n]$, $i\ne j$, add a vertex $x_{i,j}$ and
  choose independent paths $P_{i,j}^h=z_i^h\dots x_{i,j}$ of length~$r$ and
  $Q_{i,j}^h=x_{i,j}\dots z_j^h$ of length $r'-r$. Let
  \[H_{i,j}=\bigcup\nolimits_{h\in[t]}P_{i,j}^h\: \cup\:
    \bigcup\nolimits_{h\in[t]}Q_{i,j}^h.\]
  See Figure~\ref{fig:Case-1} for a sketch. Set
  $Y_{i,j}=V(H_{i,j})\ssetminus\bigl(Z_i\cup Z_j\cup\{x_{i,j}\}\bigr)$, so
  $H_{i,j}[Y_{i,j}]=\bigcup_{h\in[t]}\mathring{P}_{i,j}^h\: \cup\:
  \bigcup_{h\in[t]}\mathring{Q}_{i,j}^h$. Finally, set
  $X=\{x_{i,j}\mid i,j\in[n],\:i\ne j\}$,
  $X_i=\{x_{i,j},x_{j,i}\mid j\in[n]-i\}$ and
  $Y=\bigcup_{\substack{i,j\in[n]\\i\ne j}}Y_{i,j}$. Note that
  $V(G)=X\cup Y\cup Z$.

  \begin{figure}
    \begin{center}
      \begin{tikzpicture}[scale =1]
        \def \vt{circle (1.5pt) [fill]}
        \path (-5,-1) coordinate (Q);
        \draw (Q)++(-.75,+.5) rectangle +(1,4);
        \path (Q)+(3,2.5) coordinate (xi) \vt;
        \draw (xi)+(.02,-.6) node  {$x_{i,j}$};
        \draw (Q)+(-.25,+.5) node[below] {$Z_i$};
        \draw (2.75,-.5) rectangle (1.75,3.5);
        \draw (2.25,-.5) node[below] {$Z_j$};
        \path (Q)+(1,1) node[below] {${\color{blue} P^1_{i,j}}$};
        \foreach \i in {1,2,3,4} {
          \path (Q) + (0,\i) coordinate (z\i);
          \path (Q) + (+2,\i) coordinate (x\i);
          \path (Q) + (+4,\i) coordinate (x'\i);
          \draw (z\i) -- (x\i) -- (xi)--(x'\i);
          \draw (x\i) \vt;
          \draw (z\i) node[left] {$z_i^{\i}$};
        }
        \path (2,-1) coordinate (Q);
        \path (Q) + (-2,+2.5) coordinate (xj);
        \foreach \i in {1,2,3,4} {
          \path (Q) + (0,\i) coordinate (zz\i);
          \draw (zz\i) node[right] {$z_j^{\i}$};
          \draw (zz\i)--(x'\i);
        }
        \draw[blue] (z1)--(x1)--(xi);
        \draw[red] (xi)--(x'4)--(zz4);
        \path (.5,3) node[above] {${\color{red} Q^4_{i,j}}$};
        \foreach \i in {1,2,3,4} {
          \draw (x'\i) \vt;
          \draw (z\i) \vt;
          \draw (zz\i) \vt;
          \draw (z\i) +(1,0) \vt;
          \draw (zz\i) +(-1,0) \vt;
          \draw (zz\i) +(-2,0) \vt;
        }
        \draw (xi) \vt;
      \end{tikzpicture}
      \caption{\label{fig:Case-1} A connecting graph $H_{i,j}$ for $t=4$,
        $r=3$ and $r'=7$.} 
    \end{center}
  \end{figure}
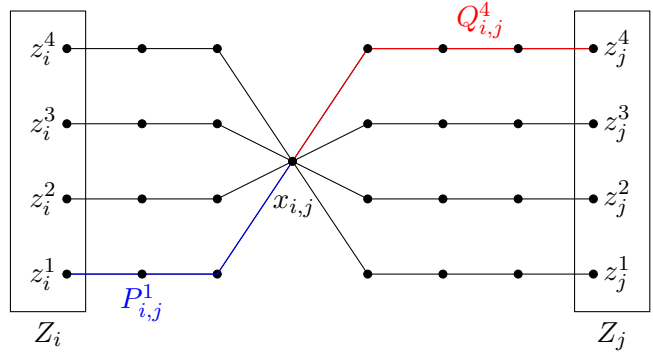

  Observe the following facts:

  {\qitem{(E1)} $\di_G(z_i^h,z_j^{h'})=r'$, for all $h,h'\in[t]$ and
    $i,j\in[n]$, $i\ne j$;

    \qitem{(E2)} every $Z$-path meets $X$, and every $Z_i,Z_j$-path with
    length $r'$ meets one of $x_{i,j},x_{j,i}$;

    \qitem{(E3)}$\di_G(x_{i,j},x)>r'$, for all
    $x\in X\ssetminus(X_i\cup X_j)$.}

  \smallskip
  The result follows from the next three claims by an easy calculation.
  
  \begin{claim}
    Let $\sigma\in\Pi(G)$ satisfy $Z<_\sigma X<_\sigma Y$. Then
    $\sco_r(G)\le\sco_r(G,\sigma)\le2t+1$.
  \end{claim}

  \begin{proof} Consider any vertex $v\in X\cup Y\cup Z$, and suppose
    $w\in S_r[G,\sigma,v]$ is witnessed by the path~$R$.

    If $v\in Z_i\subseteq Z$, then $w\le_\sigma v<_\sigma X\cup Y$, so
    $w\in Z$. By (E1) we have $w\in Z_i$, so
    $\bigl|S_r[G,\sigma,v]\bigr|\le |Z_i|=t$.

    If $v=x_{i,j}\in X$, then $Z<_\sigma v<_\sigma Y$. Thus
    $V(\mathring{R})\subseteq V(H_{i,j})\ssetminus(Z_i\cup Z_j)$, and
    $w\in Z_i\cup Z_j\cup\{x_{i,j}\}$, so
    $\bigl|S_r[G,\sigma,v]\bigr|\le2t+1$.

    If $v\in Y$, then $R\subseteq R'$ for some
    $R'\in\{P_{i,j}^h,Q_{i,j}^h\mid h\in[t]\}$. Thus
    \begin{equation}
      S_r[G,\sigma,v]\subseteq S_{r'}[G,\sigma,v]\subseteq
      \{v,v_1,v_2\},\label{eq:Y}
    \end{equation}
    where $v_1,v_2\le_\sigma v$ and $v_1,v_2\in V(R')$. The vertices
    $v_1,v_2$ exist since the ends of $R'$ come before $v$ with respect to
    $\sigma$. Thus $\bigl|S_r[G,\sigma,v]\bigr|\le3$.

    So in all cases we have $\bigl|S_r[G,\sigma,v]\bigr|\le2t+1$, hence
    $\sco_r(G,\sigma)\le2t+1$.
  \end{proof}

  \begin{claim}
    Let $\sigma\in\Pi(G)$ so that $X<_\sigma Z<_\sigma Y$. Then
    $\sco_{r'}(G)\le\sco_{r'}(G,\sigma)\le4n-6$.
  \end{claim}

  \begin{proof}
    Consider any vertex $v\in X\cup Y\cup Z$, and suppose
    $w\in S_{r'}[G,\sigma,v]$ is witnessed by the path $R$.

    If $v=x_{i,j}\in X$, then $w\le_\sigma v<_\sigma Y\cup Z$, so $w\in X$. By
    (E3) we have $w\in X_i\cup X_j$, so
    \[\bigl|S_{r'}[G,\sigma,v]\bigr|\le |X_i\cup X_j|=
      |X_{i}|+|X_{j}|-|X_{i}\cap X_{j}| =2(2n-2)-2=4n-6.\]
    If $v\in Z_i\subseteq Z$, then $X<_\sigma v<_\sigma Y$, so $w\in X$
    if~$R$ meets $X$. By (E2), $R$ meets $X_i$ if~$R$ meets
    $Z\ssetminus\{v\}$. Thus $w\in X_i\cup \{v\}$.  This gives
    $\bigl|S_{r'}[G,\sigma,v]\bigr|\le |X_i|+1=2n-1\le 4n-6$.
    
    If $v\in Y$, then $|S_{r'}[G,\sigma,v]|\le3$, by \eqref{eq:Y}.

    Thus in all cases we have $|S_{r'}[G,\sigma,v]|\le4n-6$, hence
    $\sco_{r'}(G,\sigma)\le4n-6$.
  \end{proof}

  \begin{claim}
    For any $\sigma\in\Pi(G)$, either $\sco_r(G,\sigma)\ge.246n$ or
    $\sco_{r'}(G,\sigma)\ge.754nt$.
  \end{claim}

  \begin{proof}
    Let $z_i^h$ be the $\sigma$-largest vertex of $Z$,
    $J=\{j\in[n]-i\mid z_i^h\le_\sigma V(P_{i,j}^h)\}$ and
    $\overline{J}=[n]\ssetminus J$. For all
    $j\in\overline{J}\ssetminus \{i\}$ there exists a vertex
    $u_j\in V(P_{i,j}^h)$ with $u_j<_\sigma z_i^h$; choose $u_j$ as close
    (along the path $P_{i,j}^h$) to $z_i^h$ as possible. Then
    $\{u_j\mid j\in\overline{J}\}\cup\{z_i^h\}\subseteq
    S_r[G,\sigma,z_i^h]$. Thus $\sco_r(G,\sigma)\ge|\overline{J}|+1$, and
    so we are done if $|\overline{J}|\ge.246n-1$.

    Otherwise $|J|\ge(n-1)-(.246n-1)=.754n$. For all $j\in J$ and
    $h'\in[t]$, let $v_{i,j}^{h'}$ be the vertex of $Q_{i,j}^{h'}$ with
    $v_{i,j}^{h'}<_\sigma z_i^h$ that is closest to $x_{i,j}$ (along the
    path $Q_{i,j}^{h'}$); it exists because
    $z_j^{h'}<_\sigma z_i^h<x_{i,j}$ by the choice of $z_i^h$ and the
    definition of $J$.  Then we have
    $\{v_j^{h'}\mid j\in J,$ $h'\in[t]\}\subseteq
    S_{r'}[G,\sigma,z_i^h]$. Thus $\sco_{r'}(G,\sigma)\ge t|J|\ge.754nt$.
  \end{proof}

  \noindent
  Now choose $t,n$ such that $t>1000$ and $t^\varphi\le n<
  t^\varphi+1$. Let $C=.08$. Then we have (using $\varphi<1.6181$):
  \begin{equation}
    .246 > C 2.001^\varphi\quad \text{and}\quad
    .754>C 4^\varphi.\label{eqs}
  \end{equation}
  Consider any $\sigma\in\Pi(G)$. By Claim 3, $\sco_r(G,\sigma)\ge.246n$ or
  $\sco_{r'}(G,\sigma)\ge.754nt$. In the first case, Claim 1 yields (using
  $n\ge t^\varphi$, \eqref{eqs} and $t>1000$):
  \[\sco_r(G,\sigma)\ge .246n\ge .246t^\varphi>
    C2.001^\varphi t^\varphi= C \bigl(2.001t\bigr)^\varphi>
    C(2t+1)^\varphi\ge C\bigl(\sco_r(G)\bigr)^\varphi.\]
  In the second case, Claim 2 yields (using \eqref{eqs}, $n\ge t^\varphi>
  n-1$ and $\varphi^2=\varphi+1$):
  \[\sco_{r'}(G,\sigma)\ge .754nt> C4^\varphi nt\ge
    C4^\varphi t^{\varphi+1}= C(4t^\varphi)^\varphi>
    C\bigl(4(n-1)\bigr)^\varphi>
    C\bigl(\sco_{r'}(G)\bigr)^\varphi.\qedhere\]
\end{proof}

\smallskip
\subsection{Nowhere Dense Classes and Orderings}\label{ssec2.2}

In the discussion after Corollary~\ref{cor:2} we raised the possibility of
strengthening the corollary to the following.
``\emph{A graph class $\mathcal{G}$ is \emph{nowhere dense} if and only if
  there exists a function $n^*_0:\mathbb{R}\times\NN\to\NN$ such that for
  every graph $G\in\mathcal{G}$ there exists an ordering $\sigma^*(G)$ of
  $G$ such that for every subgraph~$H$ of $G$, the ordering $\sigma^*(H)$
  of $H$ induced by $\sigma^*(G)$ has the property that for all
  $\epsilon>0$ and $r\in\NN$ such that $|H|\ge n^*_0(\epsilon,r)$ we have
  $\sco_r(H,\sigma^*(H))\le|H|^\epsilon$.}'' In this subsection we show
that such a strengthening is not possible, even for monotone nowhere dense
classes. (A class is \emph{monotone} if it closed under taking subgraphs.)

\begin{example}\label{exa:dense}\mbox{}\\*
  There exists a monotone graph class $\mathcal{G}$ that is nowhere dense
  and with the following property. There does not exist a function
  $n^*_0:\mathbb{R}\times\NN\to\NN$ such that for every graph
  $G\in\mathcal{G}$ there exists an ordering~$\sigma^*(G)$ of $G$ such that
  for every subgraph~$H$ of $G$, the ordering $\sigma^*(H)$ of $H$ induced
  by $\sigma^*(G)$ has the property that for all $\epsilon>0$ and $r\in\NN$
  such that $|H|\ge n^*_0(\epsilon,r)$ we have
  $\sco_r(H,\sigma^*(H))\le|H|^\epsilon$.
\end{example}

\begin{proof}
  Let $\mathcal{G}$ be the class of graphs whose maximum degree is at most
  their girth. (The \emph{girth} of a graph is the length of the smallest
  cycle in it.)
  Note that this class is obviously monotone. It is shown in \cite[pages
  105--106]{NO-Sparsity} that this class is nowhere dense (but not with
  bounded expansion!). One other well-known fact we use is that this class
  contains graphs with arbitrarily large minimum degree.

  Now suppose for a contradiction that there exists a function
  $n^*_0:\mathbb{R}\times\NN\to\NN$ satisfying the properties in the
  statement above. Take $0<\epsilon<1$ and $r\in\NN$, and choose an integer
  $d$ such that $d\ge n_0^*(\epsilon,r)$. Let $G$ be a graph in
  $\mathcal{G}$ with minimum degree at least~$d$. By supposition there is
  an ordering $\sigma^*(G)$ of $G$ satisfying the properties in the
  statement.

  Now let $v$ be the vertex that is last in the ordering $\sigma^*(G)$, and
  set $H=G\bigl[N_G[v]\bigr]$. Then~$H$ has at least
  $d+1>n_0^*(\epsilon,r)$ vertices. In the ordering $\sigma^*(H)$ of $H$
  induced by $\sigma^*G)$, the vertex~$v$ is still the last one, which
  gives $\sco_r(H,\sigma^*(H))= |N_G[v]|\ge d+1$. Since $|H|^\epsilon<d+1$
  for $\epsilon<1$, we cannot have $\sco_r(H,\sigma^*(H))\le|H|^\epsilon$.
\end{proof}

\section{Inspiration for the Proof of the Main Theorem}\label{sec3}

The inspiration for the proof of Theorem~\ref{thm:main} comes from the
theory of \emph{generalized game coloring numbers}, which were introduced
in \cite{KY}. In this section we define these numbers, and use a basic
result about them to give a very easy proof of a simplified version of
Theorem~\ref{thm:main}. The full proof follows in Section~\ref{sec4}.

The \emph{$r$-ordering game} is played on a graph $G$ by two players, Alice
and Bob. The game lasts for $n=|G|$ turns. The players take turns choosing
unchosen vertices with Alice playing first until there are no unchosen
vertices left. This creates an ordering $\sigma\in\Pi(G)$ of~$G$,
where~$v_i$ is the vertex chosen at the $i$-th turn and
$v_1<_\sigma v_2<_\sigma\dots<_\sigma v_n$. The \emph{score} of the game is
$\sco_r(G,\sigma)$. Alice's goal is to minimize the score while Bob's goal
is to maximize the score. The \emph{game $r$-coloring number} of $G$,
denoted $\gc_r(G)$, is the least $s$ such that Alice can always achieve a
score of at most $s$, regardless of how Bob plays.

The next result bounds the generalized game coloring numbers for any graph
class with bounded expansion.

\begin{thm}[Kierstead \& Yang \cite{KY}]\label{thm:gcolBound}\mbox{}\\*
  All graphs $G$ satisfy
  $\gc_r(G)\le 3\bigl(\wc_{2r}(G)\bigr)^2\le
  3\bigl(\sco_{2r}(G)\bigr)^{4r}$ for all $r$.
\end{thm}

Now we are ready to prove the result that inspired our general approach.

\begin{thm}\mbox{}\\*
  For any graph $G$ and $r,r'\in\NN$, there exists an ordering
  $\sigma^*\in\Pi(G)$ such that
  \[\sco_r(G,\sigma^*)\le 3\bigl(\sco_{2r}(G)\bigr)^{4r}\quad
    \text{and}\quad \sco_{r'}(G,\sigma^*)\le
    3\bigl(\sco_{2r'}(G)\bigr)^{4r'}+1.\]
\end{thm}

\medskip
\begin{proof}
  We will create the ordering by having two players A and B play the
  ordering game. Player A plays by following Alice's optimal strategy in
  the $r$-ordering game on $G$ and interprets Player B's moves as Bob's
  moves in this game. Player B ignores Alice's first move, and from then on
  plays by following Alice's optimal strategy in the the $r'$-ordering game
  on the remaining graph and interprets player A's moves as Bob's moves in
  this game.

  By Theorem~\ref{thm:gcolBound}, the resulting ordering $\sigma^*$ has the
  desired properties, where we need to be aware that Player~B had to ignore
  the first chosen vertex, which may lead to one more reachable vertex.
\end{proof}

\section{The Main Theorem}\label{sec4}

In this section we prove our main results, which are all corollaries of the
following technical theorem.

\begin{thm}\label{th-main}\mbox{}\\*
  Let $G_1,\ldots,G_k$ be a collection of graphs, all on the same vertex
  set~$V$, and $a_1,\ldots,a_k$ and $r_1,\ldots,r_k$ be positive
  integers. Set $A=a_1+\cdots+a_k$. Then there exists an ordering
  $\sigma^*$ of the common vertex set~$V$ such that for all $i=1,\ldots,k$
  we have
  \[\sco_{r_i}(G_i,\sigma^*)\le
    \frac{A}{a_i}\bigl(\wc_{2r_i}(G_i)\bigr)^2+ \wc_{2r_i}(G_i).\]
\end{thm}

\medskip
\begin{proof}
  In what follows, for a graph $G$, ordering $\sigma\in\Pi(G)$, $r\in\NN$
  and $x\in V(G)$ we use $S_r(G,\sigma,x)$ and $W_r(G,\sigma,x)$ to denote
  $S_r[G,\sigma,x]\ssetminus\{x\}$ and $W_r[G,\sigma,x]\ssetminus\{x\}$,
  respectively. We also set {\setlength{\abovedisplayskip}{3pt}
    \begin{align*}
      &V^l_\sigma(x)= \{y\in V\mid y<_\sigma x\},\quad 
        V^l_\sigma[x]=V_\sigma^l(x)\cup\{x\};\quad \text{and}\\
      &V^r_\sigma(x)=\{y\in V\mid y>_\sigma x\},\quad 
        V^r_\sigma[x]=V_\sigma^r(x)\cup\{x\}.
    \end{align*}}
  
  For all $i$, choose an ordering $\sigma_i$ of~$V$ such that
  $\wc_{2r_i}(G_i,\sigma_i)=\wc_{2r_i}(G_i)$. Define the graph $H_i$ with
  vertex set~$V$ by setting
  $E(H_i)=\{uv\mid u\in W_{r_i}(G_i,\sigma_i,v)\}$.

  \begin{claim}\label{cl4}
    For all $i$ we have $\sco_2(H_i,\sigma_i)\le\wc_{2r_i}(G_i)$.
  \end{claim}

  \begin{proof}
    If $w\in S_2(H_i,\sigma_i,v)$, then $w<_{\sigma_i}v$, and either
    $wv\in E(H_i)$ or there is a $u>_{\sigma_i}v$ with $vu,uw\in E(H_i)$.
    In the first case we have
    $w\in W_{r_i}(G_i,\sigma_i,v)\subseteq W_{2r_i}(G_i,\sigma_i,v)$. In
    the second case there are paths $P=v\dots u$ and $Q=u\dots w$ in $G_i$
    of length at most $r_i$ with
    $v\le_{\sigma_i}V(P\cup Q)\ssetminus\{w\}$. This again gives
    $w\in W_{2r_i}(G_i,\sigma_i,v)$.
  \end{proof}

  \noindent
  We construct $\sigma^*$ one vertex at the time, by \emph{collecting} one
  by one vertices from~$V$. Each time a vertex is collected it is deleted
  from the set $U$ of uncollected vertices and put at the end of the
  initial segment of~$\sigma^*$ already constructed. We maintain a vector
  $\bm{m}_v:[k]\to\{0,1,\ldots\}$ for each vertex~$v$. When
  $\bm{m}_v=\bm{0}$, we collect $v$.

  We start without any collected vertex, so $U=V$, and for all $v\in V$ and
  $i\in[k]$ we set $\bm{m}_v(i)=a_i$. We now run the following algorithm.

  \medskip
  \begin{algorithmic}[1]
    \STATE pick any $v\in U$;
    \WHILE {$U\ne\emptyset$}
    \STATE pick any $i\in [k]$ with ${\bm{m}_v(i)\ne0}$;\quad
    \COMMENT{such $i$ exists, since at this point always $v\in U$}
    \STATE $\bm{m}_v(i)\:\leftarrow\:\bm{m}_v(i)-1$;
    \IF{$\bm m_v=\bm0$}
    \STATE collect $v$
    \ENDIF;
    \IF {$N_{H_i}[v]\cap U\ne \varnothing$}\label{L8}
    \STATE $v\:\leftarrow\:\sigma_i$-$\min(N_{H_i}[v]\cap U)$\label{L9}
    \ELSIF {$U\ne \varnothing$}
    \STATE pick any $v\in U$
    \ENDIF;
    \ENDWHILE;
  \end{algorithmic}

  \medskip
  \begin{claim}\label{cl5}
    At any time in the algorithm and for all $i\in[k]$, every uncollected
    vertex $w$ satisfies: the number of collected vertices in
    $N_{H_i}(w)\cap V_{\sigma_i}^r(w)$ is at most
    $\dfrac{A}{a_i}\wc_{2r_i}(G_i)$. In other words,
    \begin{equation}
      \bigl|N_{H_i}(w)\cap V_{\sigma_i}^r(w)\cap V^l_{\sigma^*}(w)\bigr|\le
      \frac{A}{a_i}\wc_{2r_i}(G_i).\label{eq:c5}
    \end{equation}
  \end{claim}

  \medskip
  \begin{proof}
    We say that a vertex is \emph{processed} when it plays the role of $v$
    at Line~3 of the algorithm. Observe that each vertex is processed on
    exactly $A$ rounds---on $\bm{m}_v(i)=a_i$ rounds with each index
    $i\in[k]$---before it is collected at Line 6, and then it is never
    processed again.
  
    Suppose $w$ is uncollected at Line~2 of some round of the
    algorithm. Let $s$ be the number of collected vertices $v$ in
    $N_{H_i}(w)\cap V_{\sigma_i}^r(w)$. On each round that such a
    vertex~$v$ was processed with index $i$, the if-clause at Line~8 was
    witnessed by $w$. As~$w$ is uncollected at Line~2, there were at most
    $A$ rounds on which $w$ was chosen at Line~1 or Line~9 to be processed
    next. (If equality holds, then $w$ is the last vertex chosen at Line~9
    of the previous round.)  On all other such rounds, a vertex
    $w'\in N_{H_i}[v]\cap U$ with $w'<_{\sigma_i}w$ was picked to be
    processed next. Clearly, $w'\in S_2(H_i,\sigma_i,w)$. Moreover, as
    $w'\in U$, it is chosen on at most $A$ rounds.

    So all in all we get that
    $s\cdot a_i\le A+A\cdot\bigl|S_2(H_i,\sigma_i,w)\bigr|=
    A\cdot\bigl|S_2[H_i,\sigma_i,w]\bigr|$. Using Claim~\ref{cl4} this
    gives
    \[s\le \frac{A}{a_i}\sco_2[H_i,\sigma_i,w]\le
      \frac{A}{a_i}\wc_{2r_i}(G_i)\]
    as claimed.
  \end{proof}

  \noindent
  Let $\sigma^*$ be the ordering obtained by the algorithm. Take
  $i\in[k]$. We will bound $\bigl|S_{r_i}[G_i,\sigma^*\!,w]\bigr|$ for each
  $w\in V$. First notice that $S_{r_i}[G_i,\sigma^*\!,w]$ is determined at
  the moment~$w$ is collected (since then the sets $V^l_{\sigma^*}[w]$ and
  $V_{\sigma^*}^r[w]$ are known).
  
  For all $u\in S_{r_i}(G_i,\sigma^*\!,w)$, pick a path $P_u=u\dots w$ in
  $G_i$ of length at most $r_i$ with
  $V(\mathring{P}_u)\subseteq V_{\sigma^*}^r(w)$. Let
  $p_u=\sigma_i$-$\min(V(P_u))$. Then
  \begin{equation}\label{eq:4}
    \text{\hbox to 1.5\parindent{(a)\hss}} u<_{\sigma^*}w\qquad
    \text{and}\qquad
    \text{\hbox to 1.5\parindent{(b)\hss}} p_u\le_{\sigma_i}u.
  \end{equation}
  Partition $S_{r_i}(G_i,\sigma^*\!,w)$ by:
  \begin{align*}
    &X_1=\{u\in S_{r_i}(G_i,\sigma^*\!,w)\bigm| p_u=u\},\\
    &X_2=\{u\in S_{r_i}(G_i,\sigma^*\!,w)\bigm| p_u=w\}\quad
      \text{and}\\
    &X_3=\{u\in S_{r_i}(G_i,\sigma^*\!,w)\bigm|p_u<_{\sigma_i}\{u,w\}\}.
  \end{align*}
  
  If $u\in X_1$, then $P_u$ witnesses that $u\in
  W_{r_i}(G_i,\sigma_i,w)$. By the choice of $\sigma_i$ this gives
  $|X_1|\le\wc_{r_i}(G_i)-1\le\wc_{2r_i}(G_i)-1$.

  Next consider a vertex $u\in X_2$. Then $w=p_u\le_{\sigma_i}V(P_u)$, and
  hence $w\in W_{r_i}[G_i,\sigma_i,u]$. By definition, $uw\in E(H_i)$. On
  the other hand, $u<_{\sigma^*}w$ by~(4a).
  Thus we have
  $X_2\subseteq N_{H_i}(w)\cap V_{\sigma_i}^r(w)\cap V^l_{\sigma^*}(w)$.
  By \eqref{eq:c5} this means $|X_2|\le\dfrac{A}{a_i}\wc_{2r_i}(G_i)$.

  Finally, consider a vertex $u\in X_3$. Then
  $p_u\in W_{r_i}(G_i,\sigma_i,u)$ and $p_u\in W_{r_i}(G_i,\sigma_i,w)$. By
  definition, $p_uu\in E(H_i)$. By~(4a), $u<_{\sigma^*}w$, and by~(4b),
  $p_u<_{\sigma_i}u$. Combining this all gives
  $u\in N_{H_i}(p_u)\cap V_{\sigma_i}^r(p_u)\cap V^l_{\sigma^*}(p_u)$. It
  follows that
  \[X_3\subseteq \bigcup_{p\in W_{r_i}(G_i,\sigma_i,w)}N_{H_i}(p)\cap
    V_{\sigma_i}^r(p)\cap V^l_{\sigma^*}(p).\]
  And so \eqref{eq:c5} leads to
  \[|X_3|\le \bigl(\wc_{2r_i}(G_i)-1\bigr)\cdot
    \frac{A}{a_i}\wc_{2r_i}(G_i).\]

  Adding it all together we get
  \begin{align*}
    \bigl|S_{r_i}[G_i,\sigma^*\!,w]\bigr|&=1+|X_1|+|X_2|+|X_3|\\[1mm]
    &\hspace{-5mm}\le 1+\bigl(\wc_{2r_i}(G_i)-1\bigr)+
      \frac{A}{a_i}\wc_{2r_i}(G_i)+ \bigl(\wc_{2r_i}(G_i)-1\bigr)\cdot
      \frac{A}{a_i}\wc_{2r_i}(G_i)\\[1mm]
    &\hspace{-5mm}=\frac{A}{a_i}\bigl(\wc_{2r_i}(G_i)\bigr)^2+
      \wc_{2r_i}(G_i).
  \end{align*}
  Since
  $\sco_{r_i}(G_i,\sigma^*)= \max\limits_{w\in
    V}\bigl|S_{r_i}[G_i,\sigma^*\!,w]\bigr|$, the theorem follows.
\end{proof}

\noindent
We are now ready to prove the results stated in Subsection~\ref{ssec1.3}.
We start with the easiest proof.

\begin{proof}[Proof of Theorem~\ref{thm:graphU}]
  Let $G_1,\ldots,G_k$ and $r_1,\ldots,r_k\in\NN$ as in the statement of
  the theorem. Using Theorem~\ref{th-main} with all $a_i=1$, and hence
  $A=k$, we get that there exists an ordering $\sigma^*$ of~$V$ such that
  for all $i$ we have
  \[\sco_{r_i}(G_i,\sigma^*)\le
    k\cdot\bigl(\wc_{2r_i}(G_i))\bigr)^2+\wc_{2r_i}(G_i)\le
    (k+1)\bigl(\wc_{2r_i}(G_i)\bigr)^2.\qedhere\]
\end{proof}

\medskip
\begin{proof}[Proof of Theorem~\ref{thm:main}]
  Set $n=|G|$. It is easy to check that the result holds if $n\le3$, so
  assume $n\ge4$ and let $k=\bigl\lfloor\log_2(n-2)\bigr\rfloor$.

  If $i\ge k+1$, then we have $i>\log_2(n-2)$, hence $2^i+1>n-1$. This
  means that $\sco_i(G,\sigma^*)\le(2^i+1)\cdot\bigl(\wc_{2i}(G)\bigr)^2$
  trivially holds for any ordering $\sigma^*$.

  For $i=1,\ldots,k$, set $G_i=G$, $r_i=i$ and $a_i=2^{k-i}$. Then
  $A=a_1+\cdots+a_k=2^k-1$. Using Theorem~\ref{th-main}, we find that there
  exists an ordering~$\sigma^*$ of $G$ such that for all $i=1,\ldots,k$ we
  have
  \begin{align*}
    \sco_i(G,\sigma^*)
    &\le \frac{(2^k-1)\cdot\bigl(\wc_{2i}(G)\bigr)^2}{2^{k-i}}
      +\wc_{2i}(G)\\[1mm]
    &\le 2^i\cdot\bigl(\wc_{2i}(G)\bigr)^2+\wc_{2i}(G)\le
      (2^i+1)\cdot\bigl(\wc_{2i}(G)\bigr)^2.
  \end{align*}
  By \eqref{eq1} this proves the bound on $\sco_i(G,\sigma^*)$ for
  $i\le k$, and completes the proof.
\end{proof}

\noindent
We finish with a more general version of Theorem~\ref{thm:main}.

\begin{cor}\label{cor-3}\mbox{}\\*
  For any graph~$G$ and $\epsilon>0$, there exits an ordering $\sigma^*$
  of~$G$ such that for all $r\in\NN$ we have
  \[\sco_r(G,\sigma^*)\le
    \Bigl(\frac{(1+\epsilon)^{r+1}}{\epsilon^2}+1\Bigr)\cdot
    \bigl(\sco_{2r}(G)\bigr)^{4r}.\]
\end{cor}

\medskip
\begin{proof}
  We follow the proof of Corollary~\ref{thm:main} above. First choose the
  positive integer~$k$ such that
  \[\Bigl(\dfrac{(1+\epsilon)^{(k+1)+1}}{\epsilon^2}+1\Bigr)\ge|G|.\]
  This means that the bound on $\sco_i(G,\sigma^*)$ trivially holds for
  $r\ge k+1$, for any ordering $\sigma^*$.

  Now for $i=1,\ldots,k$, set $G_i=G$, $r_i=i$ and
  $a_i=\bigl\lceil(1+\epsilon)^{k+1-i}-1\bigr\rceil$. Then we can estimate
  \[A= a_1+\cdots+a_k\le \sum_{i=1}^{k}(1+\epsilon)^{k+1-i}=
    \frac{(1+\epsilon)^{k+1}-(1+\epsilon)}{\epsilon}<
    \frac{(1+\epsilon)^{k+1}}{\epsilon}.\]
  For all $i=1,\ldots,k$ we get
  \[a_i= \bigl\lceil(1+\epsilon)^{k+1-i}-1\bigr\rceil\ge
    (1+\epsilon)^{k+1-i}-1> \epsilon\cdot(1+\epsilon)^{k-i}.\]
  Using Theorem~\ref{th-main} again, there exists an ordering~$\sigma^*$ of
  $G$ such that for all $i=1,\ldots,k$ we have
  \begin{align*}
    \sco_i(G,\sigma^*)
    &\le  \frac{(1+\epsilon)^{k+1}\cdot\bigl(\wc_{2i}(G)\bigr)^2}%
      {\epsilon^2\cdot(1+\epsilon)^{k-i}}+\wc_{2i}(G)\\[1mm]
    &\le \frac{(1+\epsilon)^{i+1}}{\epsilon^2}\cdot
      \bigl(\wc_{2i}(G)\bigr)^2+\wc_{2i}(G)\le
    \Bigl(\frac{(1+\epsilon)^{i+1}}{\epsilon^2}+1\Bigr)\cdot
      \bigl(\wc_{2i}(G)\bigr)^2.
  \end{align*}
  By \eqref{eq1} this proves the bound on $\sco_i(G,\sigma^*)$ for
  $i\le k$, and completes the proof.
\end{proof}

\section{Algorithmic Aspects}

Our main results, Theorems~\ref{thm:main} and \ref{th-main}, guarantee the
existence of a specific ordering of the vertices of a graph. But the
results do not indicate if such an ordering can be found efficiently. The
proof of Theorem~\ref{th-main} is in fact algorithmic. If for every
$i=1,\ldots,k$ we have an ordering~$\sigma_i$ of the vertex set such that
$\wc_{2r_i}(G_i,\sigma_i)=\wc_{2r_i}(G_i)$, then the proof gives an
algorithm that finds an ordering $\sigma^*$ in $O(A\cdot|V|)$ steps. (We
start with a vector $\bm{m}$ with $\bm{m}_v(i)=a_i$ for each vertex $v$,
and in each iteration of the while loop one coordinate $\bm{m}_v$ gets
reduced by one.)

So the question about the existence of an efficient algorithm to find a
uniform ordering depends on the existence of an efficient algorithm to find
optimal orderings for the generalized coloring numbers. It is very unlikely
that this is possible, though. Grohe et al.\ \cite{GKRSS,GKRSSj} proved
that computing $\wc_r(G)$ is NP-complete for all fixed $r\ge3$. Note that
calculating the coloring number $\co(G)$ can be done in polynomial time; it
is an interesting open problem to determine the computational complexity
status of finding $\wc_2(G)$.

Nevertheless, it is possible to find orderings that \emph{approximate} the
generalized coloring numbers, using ideas developed in Dvo{\v
  r}\'ak~\cite{Dv:CF}. We need a new concept. Let $r\in\NN$. For a
graph~$G$, ordering $\sigma\in\Pi$ and $x\in V$, let $b_r[G,\sigma,x]$ be
the maximum number of paths of length at most~$r$ that have $x$ as one end,
whose other end~$y$ satisfies $y\le_\sigma x$, and that are vertex-disjoint
apart from $x$. Clearly, we can assume that the internal vertices of the
paths appear after $x$ in the ordering. The \emph{$r$-admissibility
  of~$G$}, denoted $\adm_r(G)$, is defined as\,\footnote{\,The definition
  of $\adm_r(G)$ in~\cite{Dv:CF} does not include the vertex~$x$ in the set
  $b_r[G,\sigma,x]$; we include it here for consistency with the now
  standard convention for generalized coloring numbers.}
\[\adm_r(G,\sigma)=\max_{x\in V}b_r[G,\sigma,x];\quad
  \adm_r(G)=\min_{\sigma\in\Pi}\,\adm_r(G,\sigma).\]
It is obvious that once again $\adm_1(G)$ is just the coloring number
$\co(G)$; while we also have $\adm_r(G)\le\sco_r(G)\le\wc_r(G)$. On the
other hand, Dvo{\v r}\'ak~\cite[Lemma~6]{Dv:CF} gives the existence of a
function $F:\NN\times\NN\to\NN$ such that
$\wc_r(G)\le F\bigl(r,\adm_r(G)\bigr)$ for all $r\in\NN$ and graphs~$G$.

Dvo{\v r}\'ak~\cite{Dv:CF} also gives a simple algorithm that, given
$r\in\NN$ and a graph~$G$, in $O(r^3\cdot|G|)$ steps finds an ordering
$\sigma$ of~$G$ such that $\adm_r(G,\sigma)\le r\cdot\adm_r(G)$.

Combining all this with the proof of Theorem~\ref{th-main} gives the
following algorithmic version of that theorem.

\begin{thm}\label{th-maina}\mbox{}\\*
  There exists a function $\phi:\NN\times\NN\to\NN$ and an algorithm
  $\mathcal{A}$ such that the following holds. Let $G_1,\ldots,G_k$ be a
  collection of graphs, all on the same vertex set~$V$, and
  $a_1,\ldots,a_k$ and $r_1,\ldots,r_k$ be positive integers. Set
  $A=a_1+\cdots+a_k$. Then algorithm $\mathcal{A}$ gives an ordering
  $\sigma^*$ of the common vertex set~$V$ such that for all $i=1,\ldots,k$
  we have
  \[\sco_{r_i}(G_i,\sigma^*)\le
  \frac{A}{a_i}\cdot\phi\bigl(r_i,\wc_{2r_i}(G_i)\bigr).\]
  The number of steps algorithm $\mathcal{A}$ requires is polynomial in $A$
  and $|G|$.
\end{thm}

\noindent
The proof of Theorem~\ref{thm:main} shows that we can use the theorem above
with $A\le n$ to get an algorithmic version of that theorem.

\begin{thm}\label{thm:maina}\mbox{}\\*
  There exists a function $\phi':\NN\times\NN\to\NN$ and an algorithm
  $\mathcal{A}'$ such that the following holds. For any graph~$G$,
  algorithm $\mathcal{A}'$ gives an ordering~$\sigma^*$ of~$G$ such that
  $\sco_r(G,\sigma^*)\le\phi'\bigl(r,\sco_{2r}(G)\bigr)$ for all
  $r\in\NN$. The number of steps algorithm $\mathcal{A}'$ requires is
  polynomial in~$|G|$.
\end{thm}

\medskip
\noindent
Finally, we formulate an algorithmic version of Corollary~\ref{cor:1}.

\begin{cor}\label{cor:1a}\mbox{}\\*
  There exists an algorithm $\mathcal{A}^*$ such that the following holds.
  A graph class $\mathcal{G}$ has bounded expansion if and only if there
  exists a function $c^*:\NN\to\NN$, such that for every graph
  $G\in\mathcal{G}$, algorithm $\mathcal{A}^*$ gives an ordering $\sigma^*$
  of $G$ such that $\sco_r(G,\sigma^*)\le c^*(r)$ for all $r\in\NN$. The
  number of steps algorithm $\mathcal{A}^*$ requires is polynomial
  in~$|G|$.
\end{cor}

\section{Discussion}

The original motivation in \cite{KY} for defining generalized coloring
numbers was to study various game theoretic questions, including
generalized game coloring numbers and their applications to other games. It
was a major surprise that generalized coloring numbers could provide
characterizations of sparse classes; indeed even generalized game coloring
numbers provide these characterizations. Just as ordinary coloring numbers
have proved useful in sparsity theory, one might expect that game coloring
numbers should find applications. Prior to this paper, and aside from the
characterization just mentioned, we know only one other application to a
non-game problem. In \cite{KK-Pack}, the game strong $2$-coloring number is
used to provide improved bounds for Bollob\'as-Eldridge-type questions on
packing. In this paper, while we used game coloring techniques, we did not
apply any theorems from that area. We limited the competitive aspects of
the theory by enforcing a prioritization for the goals of multiple players
(graphs) using the vector $\bm{m}$. This draws on ideas from the
\emph{Harmonious Strategy} in \cite{KY2}. We expect that those ideas can be
used in other (non-game) settings as well. Other applications of the
Harmonious Strategy include \cite{KY3,YZ1,YZ2}; \cite{KY3} and \cite{YZ2}
address non-game problems.

\medskip
After solving Problem~\ref{prob:Dv}, it is natural to ask how good our
answer is. In other words: \emph{For $c:\mathbb N\to \mathbb N$, what is
  the smallest function $c^*:\mathbb N\to \mathbb N$ such that for all
  $G\in \mathcal{G}_c$ there is an ordering $\sigma^*\in \Pi(G)$ such that
  all $r\in \mathbb N$ satisfy $\sco_r(G,\sigma^*)\le c^*(r)$?}  Recall
that $\varphi=\frac12(1+\sqrt5)\approx 1.62$. Example~\ref{exa:counter} and
Theorem~\ref{thm:main} show that
\[.08c(r)^\varphi\le c^*(r)\le(2^r+1)\cdot c(2r)^{4r}.\]
The lower bound is polynomial in $c(r)$, while the upper bound is
exponential in~$c(2r)$. We don't have enough evidence to make a justified
guess on the right order of $c^*$ in terms of~$c$.

The main result in~\cite{KY}, Theorem~\ref{thm:gcolBound} in this paper,
gives an upper bound of $\gc_r(G)$ in terms of $\sco_{2r}(G)$. It is shown
in~\cite{KY} that $\gc_r(G)$ cannot be bounded in terms of
$\sco_{2r-1}(G)$. Hence it is tempting to conjecture that $c^*(r)$ cannot
be upper bounded in terms of $c(2r-1)$, but we have been unable to find
examples of graphs that confirm this.

\subsection*{Acknowledgment}

The authors thank Patrice Ossona de Mendez for suggesting the examples in
Subsection~\ref{ssec2.2}, and two anonymous referees for careful reading.

\bibliographystyle{plainurl}
\bibliography{UniversalOrderings}

\end{document}